\documentclass[reqno, 12pt]{amsart}

\newcommand{\N}{\mathbb{N}}
\newcommand{\Z}{\mathbb{Z}}
\newcommand{\R}{\mathbb{R}}
\newcommand{\C}{\mathbb{C}}

\newcommand{\T}{\mathbb{T}}

\newcommand{\bP}{\mathbb{P}}

\newcommand{\Lim}{\displaystyle\lim}

\newcommand{\Sum}{\displaystyle\sum}
\newcommand{\Int}{\displaystyle\int}

\newcommand\norm[1]{\left\lVert#1\right\rVert}

\makeatletter
\newcommand{\setword}[2]{%
  \phantomsection
  #1\def\@currentlabel{\unexpanded{#1}}\label{#2}%
}
\makeatother

\usepackage[utf8]{inputenc}
\usepackage[english]{babel}
\usepackage[T1]{fontenc}
\usepackage[margin=1.1in]{geometry}
\usepackage{amsmath,amsthm,amssymb,amsfonts}
\usepackage{blindtext}
\usepackage{epsfig}
\usepackage{multicol}
\usepackage{xcolor}
\usepackage[hidelinks]{hyperref}
\usepackage{graphicx}
\usepackage{color}
\usepackage{psfrag}
\usepackage{multirow}
\usepackage{enumitem}
\usepackage{mathtools}
\usepackage{mathptmx}
\usepackage{newtxtext,newtxmath}
\usepackage[all]{xy}
\usepackage{xfrac}
\usepackage[backend=biber]{biblatex}
\usepackage{caption}
\usepackage{subcaption}

\usepackage{comment}

\usepackage[colorinlistoftodos]{todonotes}

\def\dim{\operatorname{dim}}

\def\leb{{\operatorname{Leb}}}

\def\max{\operatorname{max}}

\def\cV{\mathcal{V}}

\def\cV{\mathcal{V}}

\def\quand{\quad\text{and}\quad}

\def\Gr{Gr}

\newtheorem{theorem}{Theorem}
\newtheorem{corollary}[theorem]{Corollary}
\newtheorem{proposition}[theorem]{Proposition}

\newtheorem{lemma}[theorem]{Lemma}

\usepackage[utf8]{inputenc} 

\def\NN{{\mathbb N}}
\def\PP{{\mathbb P}}

\def\TT{{\mathbb T}}
\def\RR{{\mathbb R}}
\def\SS{{\mathbb S}}

\def\cV{\mathcal{V}}

\def\GL{\operatorname{GL}}   
\def\Leb{\operatorname{Leb}} 
\def\dim{\operatorname{dim}} 

\def\ld{\lambda}

\begin{document}
 
\title{Analyticity of the Lyapunov exponents of random products of quasi-periodic cocycles}
\author{Jamerson Bezerra}
\address{Universidade de Lisboa, Portugal}
\email{}
\date{}

\author{Adriana Sánchez}
\address{Centro de investigación de Matemática Pura y Aplicada, Universidad de Costa Rica. San José, Costa Rica.}
\email{adriana.sanchez\_c@ucr.ac.cr}

\author{El Hadji Yaya Tall}
\address{Instituto de Matemática e Estatística, USP, São Paulo, Brazil.}
\email{eljitall@ime.usp.br}



\begin{abstract}
We show that the top Lyapunov exponent $\lambda_+(p)$ , $p = (p_1, \hdots, p_N)$ with $p_i >0$ for each $i$, associated with a random product of quasi-periodic cocycles depends real analytically on the transition probabilities $p$ whenever $\lambda_+(p)$ is simple. Moreover if the spectrum at $p$ is simple (all Lyapunov exponents having multiplicity one ) then all Lyapunov exponents depend real analytically on $p$.
\end{abstract}

\maketitle

\section{Introduction}\label{s:intro}

Let $G$ be a locally compact group and $p$ a probability measure on $G$. The theory of Lyapunov exponents studies in some sense the asymptotic behavior of a random walk on $G$, whose increments are taken to be independent with distribution $p$. This theory is well developed  when $G$ is the general linear group $\GL(V)$ on some finite dimensional vector space $V$, and sometimes is also called product of independent identically distributed (i.i.d.)  random matrices, see \cite{Bel96, BeQ16} for more details. The Lyapunov exponents in this case correspond to the exponential growth rate of the norm of the random product of the matrices. \\

One of the most important questions that comes up when studying Lyapunov exponents of linear cocycles is the \emph{regularity problem}. By this we mean the problem of understanding the behavior of the Lyapunov exponents regarding the underlying data. It is a classical result of Perés \cite{Pe91} that, for i.i.d. sequence of matrices with discrete distribution, the Lyapunov exponents are locally analytic if mild simplicity conditions are assumed.\\

The problem, of understanding how the Lyapunov exponents depend on the probability distribution is originated from the work of Furstenberg \cite{Fur63}, and has been addressed by many authors. Furstenberg, Kifer \cite{FK83} and Hennion \cite{Hen84} proved continuity of the largest lyapunov exponent at every probability distribution $p$ by assumming some form of irreducibility. Recently, Bocker and Viana \cite{BV17}, proved that continuity actually holds in all generality, for compactly supported probability distributions in $\GL(2)$. This last result was generalized by Malheiro, Viana \cite{MV15} for the Markov case, and by Backes, Brown, Butler \cite{BBB}, for a very broad setting of linear cocycles with invariant holonomies. Moreover, Avila, Eskin, Viana \cite{AEV} have announced that continuity extends to $\GL(d)$ for every dimension $d \geq 2$. For non-compactly supported probability measures, Sánchez, Viana \cite{SaV20}  proved that the Lyapunov exponents are semi-continuous with respect to the Wasserstein topology, but not with respect to the weak* topology at least in dimension 2.\\

Regarding the higher regularity problem, Le Page~\cite{LP89} proved that, in any dimension, the largest Lyapunov exponent is H\"older continuous on every compact set of probability distribution satisfying strong irreducibility and the contraction property. These assumptions can not be removed as shown in~\cite{DKS} and in~\cite{TVi}. Tall, Viana~\cite{TVi}, have proven that the Lyapunov exponents are actually H\"older continuous at every probability distribution with compact support in $\GL(2)$ whose Lyapunov spectrum is simple. And they are $\log$-H\"older at every point. About the same time, and with different approach (based on large deviations and the so-called avalanche principle) Duarte, Klein, Santos~\cite{DKS} proved that irreducibility suffices for local \emph{weak-H\"older} continuity at least in dimension 2.\\

Another important type of cocycles is the \textit{quasi-periodic cocycles}, which are measurable maps from the $m$-dimensional torus $\mathbb{T}^m$ to the space of $d\times d$ real matrices, over an ergodic torus translation $T:\mathbb{T}^m\to \mathbb{T}^m$. They are important in mathematical physics,  especially in the study of Schr\"odinger operators. It is well-known that good regularity behavior of the Lyapunov exponent can give information about the properties of the associated Schr\"odinger operator (see for example \cite{Dam_s,Dam07}). However, it is extremely hard to obtain regularity properties of the Lyapunov exponents in this context as one can see from \cite{AvJS14,DuK}. This led to the introduction of a mixed model of random product of quasi-periodic cocycles called the random-quasiperiodic cocycles (see for example \cite{CDK-Mixed}), that allow us to transfer good regularity properties (e.g. continuity, H\"older regularity...) that comes from the random realm to the quasi-periodic world.\\

With that in mind we work with the space of quasi-periodic cocycles $\mathcal{G}$ that can be equipped with a natural group structure. Consider the \emph{group of quasi-periodic cocycles} which is defined as the set of pairs $(\theta, A)$, where $\theta\in \T^m$ and $A$ is a continuous map from $\T^m$ to the space of $d$-dimensional invertible matrices, with the group operation defined by
\[
(\theta_2,A_2)\cdot (\theta_1, A_1) := (\theta_2+\theta_1, A_2(\theta_1 + \cdot) A_1(\cdot)).
\]
Denote it by $\mathcal{G}_{m,d}(\R)$ the group of quasi-periodic cocycles defined above. \\

Let $\{(\omega_n, L_n)\}_{n\in\N}$ be an i.i.d. sequence of random quasi-periodic cocycles on $\mathcal{G}_{m,d}(\R)$ with common distribution $p$,  which is a discrete measure on $\{(\theta_1,A_1),\hdots,(\theta_N,A_N)\}$, with weights $p_1,\dots, p_N$. For each $k\in \N$ we denote by
\[
(\omega^k, L^k) := (\omega_k, L_k)\cdot \hdots \cdot(\omega_1, L_1)
\]
the \emph{random product of the quasi-periodic cocycles} determined by $\{(\omega_n, L_n)\}_{n\in\N}$. The (top) \emph{Lyapunov exponent} associated with the probability vector $p\in \R^N$ is the limit 
\[
\lambda_+(p) := \Lim_{n\to\infty}\frac{1}{n}\log\norm{L^n(t)}.
\]
Which exists by the sub-additive ergodic theorem. \\

Random product of quasi-periodic cocycles has been studied by many authors. Kifer \cite{Kibook86} showed continuity of the Lyapunov exponents under irreducibility assumption. Poletti, Viana \cite{PoV19}  gave some criteria for the Lyapunov spectrum to be simple. Bezerra, Poletti \cite{BeP2019} proved that the set of random product of $k+1$ $C^r$, $0 \leq r \leq \infty$ (or analytic) quasi-periodic cocycles for which the associated largest Lyapunov exponent of the random product is positive, contains a $C^0$ open and $C^r$ dense subset which is formed by $C^0$ continuity point of the Lyapunov exponent. Ao, Duarte, Klein \cite{DKao} have announced that the largest Lyapunov exponents is H\"older continuous with respect to both probability distribution $\nu$ and the quasi-periodic cocycles in any dimension. In this work we extend Perés conclusions in \cite{Pe91} to include this class of cocycles, as established by our main theorem.\\

\begin{theorem}\label{thm:Main}
Let $p = (p_1,\dots,p_N)\in \R^N$ be a probability vector such that $p_i>0$ for every $i=1,\dots, N$. If $\lambda_+(p)$ is simple, then the function which associates each probability vector $q$ the (top) Lyapunov exponent $\lambda_+(q)$ can be extended to an analytic function in a neighborhood of $p$.
\end{theorem}

Furthermore, this result can be extended for all Lyapunov exponents, assuming simplicity of the spectrum. \\

\begin{corollary}\label{310821.8}
If additionally we assume that the Lyapunov spectrum of $p$ is simple, then for every $k\geq 1$, the function which associates each probability vector $q$ the $k$-th Lyapunov exponents $\lambda_k(p)$ can be extended to a holomorphic function.
\end{corollary}


\section{Notations and precise statements}\label{s:notations}

Let $\Sigma = \{1,\dots,N\}^{\N}$ be the set of unilateral infinite sequences of elements in $\{1,\dots,N\}$ and consider $\sigma:\Sigma \rightarrow \Sigma$ be the shift map on $\Sigma$. For each $i = 1,\dots, N$, let $f_i: \T^m\rightarrow \T^m$ denotes the translation map by a vector $\theta_i\in \T^m$.\footnote{Throughout this work we use the notation $\T^m = \R^m/\Z^m$.}\\

The \emph{base dynamics} is defined by the skew product $f:\Sigma\times\T^m\rightarrow \Sigma\times\T^m$ which is given by
\[
f(x,t) = (\sigma(x), f_{x_0}(t)),
\]
where $x = (x_n)_{n\geq 0}$.\\

For each $A = (A_1,\dots,A_N) \in C^0(\T^m;\GL_d(\R))\times\cdots\times C^0(\T^m;\GL_d(\R))$, the \emph{linear cocycle} over the base dynamics $f$ with fiber action given by $A$ is defined as the the map $F_A:\Sigma\times\T^m\times \R^d\rightarrow \Sigma\times\T^m\times\R^d$ given by
\[
F(x,t,v) = (\sigma(x), f_{x_0}(t), A_{x_0}(t)\cdot v) = (f(x,t), A_{x_0}(t)\cdot v).
\]
This map induces an action on $\Sigma\times\T^m\times\bP^{d-1}$, represented in this work by the same notation $F_A$, given by the natural projectivization of the fiber action.\\

As usual, the forward iterates of the base dynamics and the linear cocycle are denoted by the following expressions
\[
f^n(x,t) = (\sigma^n(x), f^n_x(t))\quand F^n(x,t,v) = (f^n(x,t), A^n_x(t)\cdot v).
\]
where, for $x = (x_n)_{n\geq 0}$,
\[
f^n_x(t) = f_{x_{n-1}}\circ\cdots\circ f_{x_0}(t) \quand A^n_x(t) = A_{x_{n-1}}(f_{x_{n-1}}(t))\cdots A_{x_0}(t).
\]

Denote by $\leb$ the Lebesgue measure on $\T^m$ which we identify with the Lebesgue measure on its fundamental domain $[0,1]^m$.\\

Given a probability vector $p = (p_1,\dots,p_N)\in \R^N$ we define the Bernoulli measure
\[
p = p_1\delta_1+\dots+p_N\delta_N
\]
on $\{1,\dots,N\}$.
Let $p^{\N}$ be the product probability measure on $\Sigma$ and define $\mu_p = p^{\N}\times \leb$ the product measure on $\Sigma\times \T^m$.\\

Observe that the base dynamics $f$ preserves the measure $\mu_p$. We assume throughout this work that $\theta_j$ is rationally independent for some $j=1,\dots,N$ and hence we have the ergodicity of the base system $(f,\mu_p)$. In this context, Oseledets' theorem (non invertible case) states that there exist $\kappa\in \N$ and real numbers $\lambda_1(p)>\dots>\lambda_{\kappa}(p)$ such that for $\mu_p$-a.e. $(x,t) \in \Sigma\times \T^m$ there exists a filtration of linear subspaces, called \emph{Oseledets filtration},
\[
\{0\} = E^{\kappa+1}_x(t) \subsetneq E^{\kappa}_x(t) \subsetneq \cdots \subsetneq E^2_x(t) \subsetneq E^1_x(t) = \R^d,
\]
depending measurably on $(x,t)$. Moreover, for each $i=1,\hdots,\kappa$
\[
A_x(t)(E^i_x(t)) = E^i_{\sigma(x)}(f_{x_0}(t)),
\]
and
\[
\lambda_i(A,p) = \Lim_{n\to\infty}\frac{1}{n}\log\norm{A^n_x(t)\cdot v},
\]
for every $v\in E^i_x(t)\backslash E^{i+1}_x(t)$. The numbers $\lambda_1(A,p),\dots,\lambda_{\kappa}(A,p)$ are called \emph{Lyapunov exponents}, the term \emph{Lyapunov spectrum} is used to refer to the set formed by the Lyapunov exponents. In addition, Oseledets theorem also states that, for each $i=1,\dots, \kappa$, the dimension of the linear subspace $E^i_x(t)$ does not depend on $(x,t)$ in a full measure set. The \emph{multiplicity} of a Lyapunov exponent $\lambda_i(p)$ is defined by the quantity $\dim(E^i_x(t)) - \dim(E^{i+1}_x(t))$. We say that a Lyapunov exponent $\lambda_i(A,p)$ is \emph{simple} if its multiplicity is $1$ and, in the same line, we say that the Lyapunov spectrum is \emph{simple} if each Lyapunov exponent is simple.\\

Using exterior algebra we can give an alternative characterization for the Lyapunov exponents as follows. For each $k=1,\dots, d$ and for each $i=1,\dots, N$, consider the fiber action $\wedge_k A_i:\T^m\rightarrow \GL_{\binom{d}{k}}(\R)$ given by
\[
\wedge_k A_i(t)\cdot(v_1\wedge\dots\wedge v_k) = A_i(t)\cdot v_1\wedge\dots\wedge A_i(t)\cdot v_k.
\]

Write $\wedge_k A = (\wedge A_1,\dots, \wedge A_N)$. Using the same Bernoulli measure $p$ as above and Kingman's sub-additive theorem we can guarantee that for each $k=1,\dots,d$ the limit
\[
\lambda_+(\wedge_k A, p) = \Lim_{n\to\infty}\frac{1}{n}\log\norm{(\wedge_k A)_x(t)},
\]
exists for $p^{\N}$-a.e. $x\in \Sigma$ and $\leb$-a.e. $t\in \T^m$. We also have the following formula
\[
\lambda_+(\wedge_k A, p) = \lambda_1(A,p)+\dots+\lambda_k(A,p),
\]
where each Lyapunov exponent appear in the right hand side repeatedly according with its multiplicity.\\

This gives us the aimed characterization of the Lyapunov exponents associated with the probability vector $p$ and the fiber action $A$:
\[
\lambda_k(A, p) = \lambda_+(\wedge_k A, p) - \lambda_+(\wedge_{k-1}A,p).
\]

Observe that, we also obtain 
\[
\lambda_+(A,p) = \lambda_+(\wedge_1A,p) = \lambda_1(A,p).
\]

In most part of this work we fix the fiber action
\[
A = (A_1,\dots, A_N) \in C^0(\T^m;\GL_d(\R))\times\dots\times C^0(\T^m; \GL_d(\R)),
\]
and are interested in the regularity behaviour of the map that associates each probability vector $p$ the Lyapunov exponents $\lambda_i(A,p)$. For this reason, unless it is strictly necessary, we avoid to write the dependence on the fiber action $A$ and only write $\lambda_i(p)$ for the Lyapunov exponent.\\

\section{Invariant sections}\label{sec:inv}

Let $A = (A_1,\dots,A_N) \in C^0(\T^m;\GL_d(\R))\times\cdots\times C^0(\T^m;\GL_d(\R))$ and consider the linear cocycle $F_A:\Sigma\times\T^m\times\R^d\rightarrow \Sigma\times\T^m\times\R^d$ as in section \ref{s:notations}.\\

We say that a measurable section $\cV:\T^m\rightarrow \Gr(k,d)$, where $\Gr(k,d)$ denotes the Grassmanian of the $k$-dimensional subspaces inside of $\R^d$, is a $(A,p)$-\emph{invariant section} if
\[
    A_{x}(t)\cV(t) = \cV(f_{x}(t)),
\]
for $\mu_p$-a.e. $(x,t) \in \Sigma\times\T^m$. In other terms, we have that
\[
A_j(t)\cV(t) = \cV(t+\theta_j),
\]
for  $j=1, \hdots, N$, and $\leb$-a.e. $t \in \T^m$.\\

Given an $(A,p)$-invariant section $\cV:\T^m\rightarrow \Gr(k,d)$, there exist two canonical linear cocycles that can be constructed using $\cV$. For the first construction observe that there exist measurable vector fields $\xi_1, \hdots, \xi_k$ such that for $\Leb$-almost every $t$,  $\mathfrak{B}(t) = \{\xi_1(t), \hdots, \xi_k(t)\}$ is a basis of $\cV(t)$. If $\mathfrak{E}=\{e_1, \hdots, e_k\}$ is the canonical basis of $\RR^k$ let $J_{\cV}(t)$ be the linear isomorphism between $\cV(t)$ and $\RR^k$ that sends the basis $\mathfrak{B}$ to the basis $\mathfrak{E}$. We can define a linear cocycle $F_{\cV}$ in $\Sigma \times \T^m \times \RR^k$  by
\[
    F_{\cV}(x, t, v) = (\sigma(x), f_{x_0}(t), A_{x_0}^{\cV}(t)v),
\]
where $A^{\cV}_i: \T^m\rightarrow \GL_k(\R)$ is given by
\[
A^{\cV}_i(t) := J_{\cV}(t+\theta_i)\circ A_{x_0}(t)\circ (J_{\cV}(t))^{-1},
\]
for every $i=1,\dots, N$ and $A^{\cV} = (A^{\cV}_1,\dots, A^{\cV}_N)$. For the second construction consider the factor space $\T^m \times \RR^d/\cV$ where each two points $(t, u)$ and $(t, v) \in \T^m \times \RR^d$ are identified if $u-v \in \cV(t)$. This way we obtain a vector bundle over the base $\T^m$ with fibers $\RR^d/\cV_t$. Notice that $A_x(t)$ acts on $\RR^d/\cV(t)$ for any $t$, and since $\cV$ is $(A,p)$-invariant we have that
\[
    A_x(t)(\RR^d/\cV(t)) = \RR^d/\cV(f_x(t)),
\]
for $\mu_p$-a.e. $(x,t)\in \Sigma\times\T^m$. As above, we can consider for $t\in \T^m$ a basis of $\R^d/\cV(t)$ depending measurably on $t\in \T^m$ and define a linear cocycle $F_{\R^d/\cV}$ in $\Sigma \times \T^m \times \RR^{d-k}$  by
\[
    F_{\cV}(x, t, v) = (\sigma(x), f_{x_0}(t), A_{x_0}^{\R^d/\cV}(t)v),
\]
where $A_i^{\R^d/\cV}:\T^m\rightarrow \GL_{d-k}(\R)$ is defined using the appropriate change of coordinates and $A^{\R^d/\cV} = (A^{\R^d/\cV}_1,\dots, A^{\R^d/\cV}_N)$.\\

The following Proposition is due to Furstenberg-Kifer \cite{FK60} in the case of random product of matrices, and extended by Kifer \cite{Kibook86} for random product of bundle maps. This is a crucial result since it reduces the computation of the top Lyapunov exponent of a random walk on a group of upper triangular block matrices (in particular random product in the group of reducible quasi periodic cocycles) to the top Lyapunov exponents of the random walk induced on the diagonal parts.\\

\begin{proposition}\label{lem:kifer}
Let $\mathcal{V}$ be an $(A,p)$-invariant measurable section then 
\begin{equation}\label{eq:reducible}
\ld_+(A,p) = \max\{\ld_+(A^{\cV},p), \ld_+(A^{\R^d/\cV}, p)\}    
\end{equation}
\end{proposition}
\begin{proof}
See \cite[Lemma III.3.3]{Kibook86} for the proof.
\end{proof}

We say that a pair $(A,p)$, where $A:\T^m\rightarrow \GL_d(\R)$ and $p\in \R^n$ is a probability vector, is \emph{irreducible} if it does not exists any $(A,p)$-invariant section otherwise we say that $(A,p)$ is \emph{reducible}.\\

Irreducibility is a central concept in the analysis of the Lyapunov exponents. As it can be seen in the next section it allow us to obtain finner results regarding the convergence that defines the Lyapunov exponents which is the main step in the proof of the Theorem \ref{thm:Main}. The general result will be obtained by a reduction (inductive process) to the irreducible case. We finish this section proving the continuity of the Lyapunov exponent with respect to the probability vector.\\

\begin{proposition}\label{prop:continuity}
Let $p = (p_1, \hdots, p_N)$ be a probability vector with $p_i>0$ for each $i = 1, \hdots, N$. If $p_n$ is a sequence of probability vectors converging to $p$ then 
\[
\ld_+(A,p_n) \to \ld_+(A,p) \text{ as } n \to \infty.
\]
\end{proposition}

\begin{proof}
The case in which $(A,p)$ is irreducible, was proved in \cite[Theorem IV.2.2]{Kibook86}. So, we assume that $(A,p)$ is reducible and so there exist some $d_1< d$ and a $(A,p)$-invariant section $\cV:\T^m\rightarrow \Gr(d_1,d)$. Note that the fact that $p_i>0$ for every $i=1,\dots, N$ guarantees that $\cV$ is also a $(A,p_n)$-invariant section for every $n$ sufficiently large.\\
    
By Proposition \ref{lem:kifer}, we can write
    \begin{align*}
        \lambda_+(A,p) = \max\{\lambda_+(A^{\cV},p), \lambda_+(A^{\R^d/\cV},p)\}.
    \end{align*}
Therefore, in order to finish the prove of the proposition it is enough to guarantee that
    \begin{align*}
        \Lim_{n\to\infty}\lambda_+(A^{\cV},p_n)= \lambda_+(A^{\cV}, p) \quand \Lim_{n\to\infty}\lambda_+(A^{\R^d/\cV},p_n) = \lambda_+(A^{\R^d/\cV},p).
    \end{align*}
    
We give an argument to the veracity of the first limit (for the second limit the argument is completely analogous). If the pair $(A^{\cV},p)$ is irreducible, then applying again Kifer result (\cite[Theorem IV.2.2]{Kibook86}) we have that
    \[
        \Lim_{n\to\infty}\lambda_+(A^{\cV},p_n)= \lambda_+(A^{\cV}, p).
    \]
If, otherwise, $(A^{\cV},p)$ is reducible, then we can switch the roles of $A$ and $A^{\cV}$ above and repeat the argument to build cocycles taking values in $\GL_{d_2}(\R)$ and $\GL_{d_1-d_2}(\R)$ with $d_2 < d_1$ which by Proposition \ref{lem:kifer} characterizes $\lambda_+(A^{\cV},p)$. Continuing inductively we see that this process should end (we are building a finite sequence of cocycles taking values in $\GL_{d_i}(\R)$ with decreasing $d_i$) find a irreducible pair where we can apply \cite[Theorem IV.2.2]{Kibook86} and finish the prove.
\end{proof}

\section{Hyperbolicity estimates}

In this section we extract finner properties of the fiber action $A = (A_1,\dots, A_n)$ and the Lyapunov exponent $\lambda_+(A,p)$ assuming irreducibility condition on the pair $(A,p)$. We start introducing the metric on the projective space and the quantity that measures, in average, the rate of contraction of the fiber action.\\

Consider the projective distance $d : \PP^{d-1} \times \PP^{d-1} \to [0, 1]$ defined by 
\begin{equation}\label{proj:metric}
    d(u, v) := \frac{\|u \wedge v\|}{\|u\|\|v\|} = |\sin\measuredangle(u, v)|.
\end{equation}

Set
\begin{align*}
    K_n(\alpha, p) = \sup_{\substack{u,u'\in \bP^{d-1}\\ u\neq u'}}\Int_{\Sigma\times\T^m}\left(
    \frac{d(A^n_x(t)u, A^n_x(t)u')}{d(u,u')}
    \right)^{\alpha} d\mu_p(x,t).
\end{align*}
The idea is to guarantee that under irreducibility assumption on the the pair $(A,p)$ the sequence $K_n(\alpha, p)$ decreases exponentially. This is the core feature in the proof of Theorem \ref{thm:Main} and it is precisely stated in the next proposition.\\

\begin{proposition}\label{27621.2}
Assume that $(A,p)$ is irreducible and that $\lambda_+(p)$ is simple, then there exist $\alpha_0\in (0,1)$, $\zeta>0$, $C_0>0$ and $n_0\in \N$ such that for every $n\geq n_0$ and for every $\alpha \in (0,\alpha_0)$ we have
\[
K_n(\alpha, p) \leq C_0 e^{-\zeta n}.
\]
\end{proposition}

We start our analysis by proving some properties of the sequence $(K_n(\alpha, p))_n$.
\begin{proposition}\label{310821.4}
$(K_n(\alpha,p))_{n\in\N}$ is sub-multiplicative, i.e., for every $n,l\in \N$ we have that
\[
    K_{n+l}(\alpha,p) \leq K_n(\alpha,p)K_l(\alpha,p).
\]
\end{proposition}
\begin{proof}
Note that for every $v_1,v_2\in \bP^{m-1}$, $v_1\neq v_2$ 
    \begin{align*}
        &\frac{
            d(A^{n+l}_x(t)\cdot v_1, A^{n+l}_x(t)\cdot v_2)
        }{
            d(v_1,v_2)
        } =\\
        &\frac{
            d(A^n_{\sigma^l(x)}(f^l_x(t))\cdot A^l(t)\cdot v_1, A^n_{\sigma^l(x)}(f^l_x(t))\cdot A^l(t)\cdot v_2)
        }{
            d(A^l_{x}(t)\cdot v_1, A^l_{x}(t)\cdot v_2)
        }\cdot \frac{
            d(A^l_{x}(t)\cdot v_1, A^l_{x}(t)\cdot v_2)
        }{
            d(v_1,v_2)
        }.
    \end{align*}

So, integrating with respect to $(x,t)\in \Sigma\times\T^m$ and taking the supremum we have that
    \begin{align*}
        &\Int_{\Sigma\times\T^m}\left(\frac{
            d(A^{n+l}_x(t)\cdot v_1, A^{n+l}_x(t)\cdot v_2)
        }{
            d(v_1,v_2)
        }\right)^{\alpha}\ d\mu_p(x,t) \\
        &\leq\sup_{\hat{v_1},\hat{v_2}}\Int_{\Sigma\times\T^m}\left(
            \frac{
            d(A^n_{\sigma^l(x)}(f^l_x(t))\hat{v_1}, A^n_{\sigma^l(x)}(f^l_x(t))\hat{v_2})
        }{
            d(\hat{v_1}, \hat{v_2})
        }
        \right)^\alpha\cdot\left(
            \frac{
            d(A^l_{x}(t)\cdot v_1, A^l_{x}(t)\cdot v_2)
        }{
            d(v_1,v_2)
        }
        \right)^{\alpha}d\mu_p(x,t)\\
        &\leq \sup_{\hat{v_1},\hat{v_2}}\Int_{\Sigma\times\T^m}\left(
            \frac{
            d(A^n_{x}(t)\hat{v_1}, A^n_{x}(t)\hat{v_2})
        }{
            d(\hat{v_1}, \hat{v_2})
        }
        \right)^\alpha\ d\mu_p(x,t)\cdot \Int_{\Sigma\times\T^m}\left(
            \frac{
            d(A^l_{x}(t)\cdot v_1, A^l_{x}(t)\cdot v_2)
        }{
            d(v_1,v_2)
        }
        \right)^\alpha\ d\mu_p(x,t)\\
        &\leq K_n(\alpha,p)\cdot K_l(\alpha,p),
    \end{align*}
Taking the supremum over $v_1,v_2$ the result follows.
\end{proof}

\begin{proposition}\label{27621.1}
If $(A,p)$ is irreducible and $\lambda_+(p)$ is simple, then
\begin{align*}
    \Lim_{n\to\infty}\frac 1n \Int_{\Sigma\times \T^m}\log\norm{A^n_x(t)\cdot v}\ d\mu_p(x,t) = \lambda_+(p),
\end{align*}
uniformly on $v\in \bP^{d-1}$.
\end{proposition}

The proposition \ref{27621.1} is an important consequence of the non-random filtration theory developed by Kifer in \cite{Kibook86}. In order to precisely apply non-random filtration theorem (see \cite[Theorem III.1.2]{Kibook86}) we introduce the notion of stationary measure for the base dynamics. We say that a measure $\eta$ on $\T^m$ is \emph{stationary} for $f$ if it satisfies the following equation
\begin{align*}
    \eta = \Sum_{i=1}^N \eta\circ f_i^{-1}.
\end{align*}

The next lemma is a direct consequence of the fact that the maps $f_i$ are torus translation with $\theta_i$ rationally independent for some $i$. \\

\begin{lemma}
    The Lebesgue measure is the unique stationary measure for $f$.
\end{lemma}

\begin{proof}[Proof of Proposition \ref{27621.1}]
By \cite[Theorem III.1.2]{Kibook86} there exists a full Lebesgue measure $X\subset T^m$ such that for any $t\in X$ there exist a sequence of  (non-random) linear subspaces 
    \[
        \{0\} =\cV^{k(p)+1}(t)\subset \cV^{k(p)}(t) \subset \cdots \subset \cV^1(t)\subset \cV(t)^0 = \RR^d,
    \]
    and a sequence of (non-random) values 
    \[
        -\infty < \beta_{k(p)}(p) < \cdots < \beta_1(p) < \beta_0(p)<\infty,
    \]
    with $\beta_0(p) = \lambda_+(p)$, such that for any $v \in \cV^i(t)$ 
    \[
        \lim_{n \to \infty} \frac 1n \log \|A^n_x(t).v\| = \beta_i(p),
    \]
    for each  $i= 0, \hdots k(p)$ and for $p^\NN$-almost every $x \in \Sigma$. Moreover, we have 
    \[
        A_j(t)\cV^i(t) = \cV^i(f_j(t)) \text{ for all} j= 1, \hdots N.
    \]
    
    Since we are assuming that the cocycle is irreducible, such filtration must be trivial. Therefore,
    \[
        \lim_{n \to \infty} \frac 1n \log \|A^n_x(t).v\| = \ld_+(p),
    \]
    for $p^{\N}$-a.e $x\in \Sigma$ and for every $v\in \bP^{d-1}$. By the dominated convergence theorem this implies that
    \[
        \lim_{n \to \infty} \frac 1n \int_{\Sigma \times \T^m} \log \|A_x^n(t).v\|d\mu_p(x,t) =\ld_+(p).
    \]
    
Now it remains to show that the convergence is uniform on $v \in \PP^{d-1}$. Assume that this convergence is not uniform. This implies the existence of a positive number $\epsilon > 0$ and a sequence of unit elements $v_n \in \PP^{d-1}$ such that for all large $n$,
    \begin{align}\label{310821.1}
        \frac 1n \int_{\Sigma \times \T^m}\log\|A_x^n(t).v_n \| d \mu_p (x,t) \leq \ld_+(p) - \epsilon.
    \end{align}
By compacity of $\PP^{d-1}$, we may assume that the sequence $v_n$ converges to some unit element $v \in \PP^{d-1}$. \\
    
For each $(x,t)\in \Sigma\times\T^m$ let $u_{n,i}(x,t)$ and $a_{n,i}(x,t)$ be respectively the $i$-th singular direction and the $i$-th singular value (in non-increasing order) of the matrix $A^n_x(t)$. For $\mu_p$-a.e. $(x,t)\in \Sigma\times\T^m$ we have that the accumulation points of the sequence $(u_{n,i}(x,t))_{n\in\N}$ generates the Oseledets space $E^2_x(t)$ for every $i\geq 2$. In particular, the accumulation points of $(u_{n,i})_{n\in \N}$ are orthogonal to $E^2_x(t)$ for $\mu_p$-a.e. $(x,t)\in \Sigma\times\T^m$. Also, note that
        \begin{align*}
            \mu_p\left(\left\{
                    (x,t)\in \Sigma\times\T^m;\ v\in E^2_x(t)            
                \right\}
            \right) = 0,
        \end{align*}
        which follows from the fact that for $\mu_p$-a.e. $(x,t)\in \Sigma\times\T^m$
        \begin{align*}
            \lim_{n\to\infty}\frac{1}{n}\log\norm{A^n_x(t)} = \lambda_+(p).
        \end{align*}
        
Joining this information and the fact that the multiplicity of $\lambda_1(p) = \lambda_+(p)$ is equal to $1$, we see that if $\hat{u} \in \SS^{d-1}$ is an accumulation point of $(u_{n,1})_n$ then 
        \begin{align}\label{310821.2}
            |<v,\hat{u}>| > 0.
        \end{align}
Write
        \begin{align*}
            v_n = \Sum_{i=1}^d <v_n, u_{n,i}(x,t)>\cdot u_{n,i}(x,t).
        \end{align*}

Applying $A^n_x(t)$ in both sides and using that $\{u_{n,i}\}_{i}$ are the singular directions of $A^n_x(t)$ we have that
        \begin{align*}
            \norm{A^n_x(t)\cdot v_n} \geq \norm{A^n_x(t)}\cdot |<v_n,u_{n,i}(x,t)>|.
        \end{align*}
        Hence, using \eqref{310821.2} we have that
        \begin{align}\label{310821.3}
            \liminf_{n\to\infty}\frac{
                \norm{A^n_x(t)\cdot v_n}
            }{
                \norm{A^n_x(t)}
            }> 0,
        \end{align}
        for $\mu_p$-a.e. $(x,t)\in \Sigma\times\T^m$.\\
        
Now, to conclude, observe that
        \begin{align*}
            &\frac{1}{n}\Int_{\Sigma\times\T^m}\log\norm{A^n_x(t)\cdot v_n}\ d\mu_p(x,t) \\
            &=\frac{1}{n}\Int_{\Sigma\times\T^m}\log\frac{
                \norm{A^n_x(t)\cdot v_n}
            }{
                \norm{A^n_x(t)}
            }\ d\mu_p(x,t) + \frac{1}{n}\Int_{\Sigma\times\T^m}\log\norm{A^n_x(t)}\ d\mu_p(x,t).
        \end{align*}
        
Making $n\to\infty$, using equation \eqref{310821.3} and the dominated convergence theorem we have that
        \begin{align*}
            \Lim_{n\to\infty}\frac{1}{n}\Int_{\Sigma\times\T^m}\log\norm{A^n_x(t)\cdot v_n}\ d\mu_p(x,t) = \lambda_+(p).
        \end{align*}
This contradicts the equation \eqref{310821.1} and concludes the proof of the uniform convergence.
        
\end{proof}

\begin{proposition}\label{310821.7}
There exists $n_0\in \N$ such that for every $n\geq n_0$ and for every $v_1,v_2\in \bP^{d-1}$, with $v_1\neq v_2$ we have
        \begin{align*}
            \Int_{\Sigma\times\T^m}\log
                \frac{
                    d(A^n_x(t)\cdot v_1, A^n_x(t)\cdot v_2)
                }{
                    d(v_1,v_2)
                }\ d\mu_p(x,t) < -1.
        \end{align*}
\end{proposition}

\begin{proof}
Notice that
        \begin{align*}
            \frac{
                d(A^n_x(t)\cdot v_1, A^n_x(t)\cdot v_2)
            }{
                d(v_1,v_2)
            }&= \frac{
                \norm{A^n_x(t)\cdot v_1\wedge A^n_x(t)\cdot v_2}
            }{
                \norm{v_1\wedge v_2}
            }\cdot
            \frac{
                \norm{v_1}\cdot\norm{v_2}
            }{
                \norm{A^n_x(t)\cdot v_1}\cdot\norm{A^n_x(t)\cdot v_2}
            }\\
            &\leq\frac{
                \norm{\wedge_2 A^n_x(t)}
            }{
                \norm{A^n_x(t)\cdot v_1}\cdot\norm{A^n_x(t)\cdot v_2}
            }.
        \end{align*}
Taking the logarithm in both sides, dividing by $n$ and integrating we have that
        \begin{align*}
            &\frac{1}{n}\Int_{\Sigma\times\T^m}\frac{
                d(A^n_x(t)\cdot v_1, A^n_x(t)\cdot v_2)
            }{
                d(v_1,v_2)
            }\ d\mu_p(x,t)\\
            &\leq \frac{1}{n}\Int_{\Sigma\times\T^m}\log\norm{\wedge_2 A^n_x(t)}\ d\mu_p(x,t) - \frac{1}{n}\Int_{\Sigma\times\T^m}\log\norm{A^n_x(t)\cdot v_1}\ d\mu_p(x,t) -\\ &-\frac{1}{n}\Int_{\Sigma\times\T^m}\log\norm{A^n_x(t)\cdot v_2}\ d\mu_p(x,t).
        \end{align*}
        
Therefore, using Proposition \ref{27621.1} we have that given $\varepsilon\in (0, (\lambda_1(p)-\lambda_2(p))/3)$ there exists $n_1\in \N$ (independently of $v_1$ and $v_2$) such that for every $n\geq n_1$ we have that
        \begin{align*}
            \frac{1}{n}\Int_{\Sigma\times\T^m}\frac{
                d(A^n_x(t)\cdot v_1, A^n_x(t)\cdot v_2)
            }{
                d(v_1,v_2)
            }\ d\mu_p(x,t)
            &\leq \lambda_1(p) + \lambda_2(p) + \varepsilon - (\lambda_1(p) - \varepsilon) - (\lambda_1(p) - \varepsilon)\\
            &= \lambda_2(p) - \lambda_1(p) + 3\varepsilon.
        \end{align*}
        Since $\lambda_2(p) - \lambda_1(p) - 3\varepsilon < 0$ we have that there exists $n_0\in \N$, $n_0\geq n_1$, such that
        \begin{align*}
            \Int_{\Sigma\times\T^m}\frac{
                d(A^n_x(t)\cdot v_1, A^n_x(t)\cdot v_2)
            }{
                d(v_1,v_2)
            }\ d\mu_p(x,t)
            &\leq n(\lambda_2(p) - \lambda_1(p) +3\varepsilon)\\
            & < -1,
        \end{align*}
        for every $v_1,v_2\in \bP^{d-1}$, $v_1\neq v_2$.
\end{proof}
    
Now we are finally ready to give a proof of the Proposition \ref{27621.2}.\\

\begin{proof}[Proof of Proposition \ref{27621.2}].
Observe that by sub-multiplicativity of the sequence $(K_n(\alpha, p))_n$ (see Proposition \ref{310821.4}) it is enough to guarantee that there exists some $n_0\in \N$ and $\alpha_0>0$ such that for every $\alpha\in (0,\alpha_0)$
        \begin{align}\label{310821.5}
            K_{n_0}(\alpha, p) < 1.
        \end{align}
        Since for every $n\geq n_0$ we can write $n = n_0 q + r$ with $r\in \{0,\dots, n_0-1\}$ and so
        \begin{align*}
            K_n(\alpha,p) \leq \max_{0\leq j\leq n_0-1}K_j(\alpha,p)\cdot (K_{n_0}(\alpha,p)^{1/q})^n.
        \end{align*}
        
Setting $C_0 = \max_{0\leq j\leq n_0-1}K_j(\alpha,p)$ and $\zeta = -1/q\log K_{n_0}(\alpha,p)$ the result will be concluded.\\
        
In order to guarantee \eqref{310821.5} and simplify notation, set for $v_1,v_2\in \bP^{d-1}$, $v_1\neq v_2$
        \[
            \psi_n(x,t,v_1,v_2) = \frac{
                d(A^n_x(t)\cdot v_1, A^n_x(t)\cdot v_2)
            }{
                d(v_1,v_2)
            }.
        \]
        
Using the fact that $e^t \leq 1 + t + \frac{1}{2}t^2e^{|t|}$ for every $t\in \R$, we get that
        \begin{align*}
            \Int_{\Sigma\times\T^m}&\psi_n^{\alpha}(x,t,v_1,v_2)\ d\mu_p(x,t)\\
            &= \Int_{\Sigma\times\T^m}\exp\left(
                \alpha\log\psi_n(x,t,v_1,v_2)
            \right)\ d\mu_p(x,t)\\
            &\leq 1 + \alpha\Int_{\Sigma\times\T^m}\log\psi_n(x,t,v_1,v_2)\ d\mu_p(x,t)\\ &+\frac{\alpha^2}{2}\Int_{\Sigma\times\T^m}\left(
                \log\psi_n(x,t,v_1,v_2)
            \right)^2\exp\left(\left|
                    \alpha\log\psi_n(x,t,v_1,v_2)
                \right|
            \right)\ d\mu_p(x,t).
        \end{align*}

Since the images of $\TT^m$ under the continuous maps $A_i$, $i = 1, \hdots, N$ are compact subsets of $\GL(d, \RR)$, we can find a positive constant $M$  such that for every $n\in \N$, and $\mu_p$-almost every $(x,t)\in \Sigma\times\T^m$ and every $v_1,v_2\in \bP^{d-1}$
            \begin{align}\label{310821.6}
                |\log\psi_n(x,t,v_1,v_2)| \leq n \log M.
            \end{align}

Hence, using inequality \eqref{310821.6} and taking $n_0\in \N$ given by Proposition \ref{310821.7}, we have that
        \begin{align*}
            &\Int_{\Sigma\times\T^m}\psi_n^{\alpha}(x,t,v_1,v_2)\ d\mu_p(x,t)
            \leq 1 -\alpha + \alpha^2\cdot\frac{
                n^2 (\log M)^2 M^{n}.
            }{2},
        \end{align*}
        for every $n\geq n_0$. Therefore, defining
        \begin{align*}
            \alpha_0 = \frac{2}{
                n_0^2(\log M)^2 M^{n_0}
            },
        \end{align*}
        we have that for every $\alpha\in (0,\alpha_0)$
        \begin{align*}
            K_{n_0}(\alpha,p) = \sup_{\substack{
                v_1,v_2\in \bP^{d-1}\\
                v_1\neq v_2}
            }\Int_{\Sigma\times\T^m}\psi_{n_0}^{\alpha}(x,t,v_1,v_2)\ d\mu_p(x,t) < 1.
        \end{align*}
\end{proof}

\section{Proof of the main results}

In this section we use the machinery developed in the previous section to establish the proofs of Theorem \ref{thm:Main} and Corollary \ref{310821.8}.\\
    
\paragraph{Proof of Theorem \ref{thm:Main}}
Initially we assume that the pair $(A,p)$ is irreducible. The general case will be deduced from the irreducible by an inductive argument.\\

For each $z = (z_1,\dots,z_N) \in \C^{N}$ consider the operator $T_z: C^0(\T^m\times\bP^{d-1}; \R)\rightarrow C^0(\T^m\times\bP^{d-1};\R)$ given by
    \[
        T_z\varphi(t,v) = \Sum_{i=1}^{N}z_i\varphi(f_i(t),A_i(t)\cdot v).
    \]

Note that for each $(t,v) \in \T^m\times\bP^{d-1}$, $T^n_z\varphi(t,v)$ is a homogeneous polynomial of degree $n$ in the variables $z_1,\dots,z_{N}$.\\

Let $\varphi \in C^0(\T^m\times\bP^{d-1};\R)$ be a function which is uniformly Lipschitz in each fiber $t\in \T^m$, i.e., 
\[
\sup_{t\in\T^m}Lip(\varphi(t,\cdot)) < \infty.
\]

Fix $\zeta> 0$, $n_0\in \N$ and $\alpha_0> 0$ given by Proposition \ref{27621.2}. Consider $\gamma \in (0,1)$ such that $\gamma > e^{-\zeta}$ and set
\[
D_{\gamma} = \left\{
z\in \C^{N};\ \max_i\frac{|z_i|}{p_i} < \gamma^{-1}\ \ \text{and}\ \ \Sum_{i=1}^{N}z_i = 1
\right\},
\]
which is well defined once by hypothesis we are assuming that $p_i>0$ for every $i=1,\dots, N$.\\

For each $v_1, v_2 \in \bP^{d-1}$, $v_1\neq v_2$ and $z \in D_{\gamma}$, we have
    \begin{align*}
        &\left|
            \Int_{\T^m} T_z^n\varphi(t,v_1)\ dt - \Int_{\T^m} T_z^n\varphi(t,v_2)\ dt
        \right|\\
        &\leq
        \gamma^{-n}\Int_{\Sigma\times\T^m}\left|
            \varphi(f^n_x(t), A^n_x(t)\cdot v_1) - \varphi(f^n_x(t), A^n_x(t)\cdot v_2)
        \right|\ d\mu_p(x,t)\\
        &\leq
        \gamma^{-n}\sup_{t\in\T^m} Lip(\varphi(t,\cdot))\Int_{\Sigma\times\T^m} d(A^n_x(t)\cdot v_1, A^n_x(t)\cdot v_2)\ d\mu_p(x,t).
    \end{align*}

Since, for every $\alpha> 0$,
    \begin{align*}
        d(A^n_x(t)\cdot v_1, A^n_x(t)\cdot v_2) \leq \left(
            \frac{d(A^n_x(t)\cdot v_1, A^n_x(t)\cdot v_2)}{d(v_1,v_2)}
        \right)^{\alpha},
    \end{align*}
    we have that
    \begin{align*}
        &\left|
            \Int_{\T^m} T_z^n\varphi(t,v_1)\ dt - \Int_{\T^m} T_z^n\varphi(t,v_2)\ dt
        \right|\leq
        \gamma^{-n}\sup_{t\in\T^m} Lip(\varphi(t,\cdot))K_n(\alpha,p).
    \end{align*}

Hence taking $\alpha\in (0,\alpha_0)$ and using Proposition \ref{27621.2}, we have
\begin{align}\label{24521.0}
    &\left|
    \Int_{\T^m} T_z^n\varphi(t,v_1)\ dt - \Int_{\T^m}T_z^n\varphi(t,v_2)\ dt
    \right|
    \leq
    C_0\sup_{t\in\T^m}Lip(\varphi(t,\cdot)) \gamma^{-n}e^{-\zeta n}.
\end{align}

Now, for $z \in D_{\gamma}$ and $\varphi$ as above and any $v\in \bP^{d-1}$, using \eqref{24521.0}, we have, for every $n\geq n_0$, that
\begin{align*}
    &\left|
    \Int_{\T^m}T^{n+1}_z\varphi(t,v)\ dt - \Int_{\T^m} T^n_z\varphi(t,v)\ dt
    \right|=\\
    &\hspace{3cm}=
    \left|
    \Sum_{i=1}^{N}z_i\left(
    \Int_{\T^m}T^n_z\varphi(f_i(t),A_i(t)\cdot v)\ dt - \Int_{\T^m}T^n_z\varphi(t,v)\ dt
    \right)
    \right|\\
    &\hspace{3cm}=
    \left|
    \Sum_{i=1}^{N}z_i\left(
    \Int_{\T^m}T^n_z\varphi(t,A_i(f_i^{-1}(t))\cdot v)\ dt - \Int_{\T^m}T^n_z\varphi(t,v)\ dt
    \right)
    \right|\\
    &\hspace{3cm}\leq
    \Sum_{i=1}^{N}|z_i|\left|
    \Int_{\T^m}T^n_z\varphi(t,A_i(f_i^{-1}(t))\cdot v)\ dt - \Int_{\T^m}T^n_z\varphi(t,v)\ dt
    \right|\\
    &\hspace{3cm}\leq
    N\gamma^{-1} C_0\sup_{t\T^m}Lip(\varphi(t,\cdot))\gamma^{-n}e^{-n\zeta}.
\end{align*} 

Hence, taking $m= n+\ell>n$ we have that
\begin{align*}
    &\left|
    \Int_{\T^m} T^m_z\varphi(t,v)\ dt - \Int_{\T^m} T^n_z\varphi(t,v)\ dt
    \right|\leq\\
    &\hspace{3cm}\leq
    \Sum_{j=1}^{\ell}\left|
    \Int_{\T^m} T^{n+j}_z\varphi(t,v)\ dt - \Int_{\T^m} T^{n+j-1}_z\varphi(t,v)\ dt
    \right|\\
    &\hspace{3cm}\leq
    C_0N\gamma^{-1}\sup_{t\in\T^m}Lip(\varphi(t,\cdot))\Sum_{j=1}^{\ell}(e^{-\zeta}\gamma^{-1})^{n+j-1}.
\end{align*}

Then,
\begin{align*}
    &\left|
    \Int_{\T^m} T^m_z\varphi(t,v)\ dt - \Int_{\T^m} T^n_z\varphi(t,v)\ dt
    \right|
    \leq\\
    &\hspace{3cm}\leq C_0\sup_{t\in\T^m}Lip(\varphi(t,\cdot))N\gamma^{-1}(e^{-\zeta}\gamma^{-1})^n\Sum_{j=1}^{\ell}(e^{-\zeta}\gamma^{-1})^{j-1}.
\end{align*}

Therefore, for each $v\in \bP^{d-1}$, the sequence of holomorphic functions
\[
\left\{
z\longmapsto \Int_{\T^m} T^n_z\varphi(t,v) dt;\ n\geq 0
\right\},
\]
converges uniformly in the set
\[
D_{\gamma} = \left\{
z\in \C^{N};\ \Sum_{i=1}^{N}z_i = 1, \max_{1\leq j\leq N}\frac{|z_j|}{p_j}<\gamma^{-1}
\right\}.
\]

This implies in particular that for each $v\in \bP^{d-1}$ the function
    \begin{align*}
        z\longmapsto \lim_{n\to\infty} \Int_{\T^m}T^n_z\varphi(t,v)\ dt
    \end{align*}
is holomorphic.

For each $j = 1,\dots, N$, consider the functions
    \[
        \varphi_j(t,v) = \log\frac{\norm{A_j(t)\cdot v}}{\norm{v}}.
    \]

Define, for each $n\geq 1$ and each probability vector $q\in D_{\gamma}\cap\R^N$,
    \[
        \xi_n(t,v,q) = \Sum_{j=1}^{N}q_jT^n_q\varphi_j(t,v) =
        \Int_{\Sigma}\log\frac{\norm{A_x^{n+1}(t)\cdot v}}{\norm{A_x^n(t)\cdot v}} d q^{\N}(x),
    \]
and observe that
    \[
        \frac{1}{\ell}\Sum_{n=0}^{\ell-1}\xi_n(t,v,q) = \frac{1}{\ell}\Int_{\Sigma}\log\norm{A_x^{\ell}(t)\cdot v}d q^{\N}(x) - o_{\ell}(1),
    \]
which implies that 
        \begin{align*}
            \Lim_{\ell\to\infty} \frac{1}{\ell}\Sum_{n=0}^{\ell-1}\Int_{\T^m}\xi_n(t,v,q)\ dt &= \Lim_{n\to\infty}\frac{1}{n}\Int_{\Sigma\times\T^m}\log\norm{A_x^n(t)\cdot v}\ d\mu_q(x,t)
        \end{align*}

The (top) Lyapunov exponent being continuous (proposition \ref{prop:continuity}), then $\lambda_+(q)$ is also simple. Since irreducibility is an open property, we can apply  Proposition \ref{27621.1} to $q$ to conclude that the  limit of the right hand side of the above equation must be equal to $\lambda_+(q)$.\\

Observe that the sequence
    \begin{align*}
        \left(
            \Int_{\T^m}\xi_n(t,v,q)\ dt
        \right)_n = \left(
            \Sum_{j=1}^N q_j\Int_{\T^m} T^n_q\varphi_j(t,v)\ dt
        \right)_n,
    \end{align*}
converges and its Cesàro sum
    \begin{align*}
        \left(
            \frac{1}{\ell}\Sum_{n=0}^{\ell-1}\Int_{\T^m}\xi_n(t,v,q)\ dt
        \right)_n,
    \end{align*}
converges to $\lambda_+(q)$ we conclude that for every $v\in \bP^{d-1}$
    \begin{align*}
        \Lim_{n\to\infty}\Sum_{j=1}^N q_j\Int_{\T^m} T^n_q\varphi_j(t,v)\ dt = \lambda_+(q).
    \end{align*}

But as was observed before the right hand side of the above equation has a holomorphic extension to the domain $D_{\gamma}$. This shows that the function $q\longmapsto \lambda_+(q)$ has a holomorphic extension to the $D_{\gamma}$ concluding the proof of the Theorem assuming the irreducibility of $(A,p)$.\\

Now we drop the assumption that  $(A,p)$ is irreducible. In this case, there exists a $(A,p)$-invariant section $\cV:\T^m\rightarrow \Gr(d_1,d)$, with $d_1<d$. Applying lemma \eqref{lem:kifer} we obtain that 
\[
    \ld_+(p, F) = \max\{\ld_+(A^{\cV},p), \ld_+(A^{\R^d/\cV},p))\}.
\]

Using the fact that the $\lambda_+(p)$ is simple, we conclude that 
\[
    \ld_+(A^{\cV}, p) \neq \ld_+(A^{\R^d/\cV}, p).
\]

Assume without loss of generality that $\lambda_+(A,p) = \lambda_+(A^{\cV},p) >\lambda_+(A^{\R^d/\cV},p)$. This implies that for every probability vector $q\in \R^N$ in a neighborhood of $p$ we have that
    
    \begin{align*}
        \lambda_+(A,q) = \lambda_+(A^{\cV}, q)
    \end{align*}
    (remember that $(A,p)$ and $(A,q)$ share the same invariant section which is $\cV$).\\
    
If $(A^{\cV},p)$ is irreducible, then, by the previous part, we have that the function $q\longmapsto \lambda_+(A^{\cV},q)$ has an holomorphic extension to a complex domain $D_{\gamma}$. Since $\lambda_+(A,q)$ coincides with $\lambda_+(A^{\cV},q)$ in a neighborhood of $p$ we have that in this neighborhood $q\longmapsto \lambda_+(A,q)$ has a holomorphic extension.\\
    
However, if $(A^{\cV},p)$ is reducible, then there exists a $(A^{\cV},p)$-invariant section $V_1:\T^m\rightarrow \Gr(d_2,d_1)$ with $d_2< d_1$ (if $\lambda_+(A,p) = \lambda_+(A^{\R^d/\cV},p)$ above then we change $d_1$ for $d-d_1$) and so we can switch the roles of $A$ and $A^{\cV}$ to conclude that there exists a neighborhood of $p$ such $q\longmapsto\lambda_+(A,q)$ has an holomorphic extension.\\
    
The analysis is concluded observing that the process must finish once we have that the invariant section in each process are taking values in subspaces whose dimension are decreasing ($d>d_1 > d_2> \dots$). This finishes the proof of the Theorem \ref{thm:Main}.
    
\paragraph{Proof of Corollary \ref{310821.8}:} 
As described in the end of Section \ref{s:notations} we see for any $1\leq k\leq d$
    \begin{align*}
        \lambda_+(\wedge_kA, p) = \lambda_1(p)+\dots+\lambda_k(p).
    \end{align*}
The fact that the Lyapunov spectrum of $p$ is simple guarantees that the top Lyapunov exponent $\lambda_+(\wedge_k A,p)$ is simple. Then, for each $k\in\{1,\dots,d\}$ we can apply Theorem \ref{thm:Main} obtaining that the map
    \begin{align*}
        q\in\R^N\longmapsto \lambda_+(\wedge_kA,q),
    \end{align*}
where $q$ is taken among the probability vectors, has a holomorphic extension to some complex domain around $p\in \R^N$. Therefore, the maps
    \begin{align*}
        q\longmapsto \lambda_k(q) = \lambda_+(\wedge_kA,q) - \lambda_+(\wedge_{k-1}A,q),
    \end{align*}
have a holomorphic extension to a complex domain around $p\in \R^N$ concluding the proof of the corollary.

\section*{Acknowledgments}
The authors thank Ali Tahzibi for suggesting this problem. This work was partially supported by Universidad de Costa Rica for A.S. and by FAPESP grant \#18/07797-5 for Y.T. J. B. was supported by FCIENCIAS.

\printbibliography

\end{document}